\documentclass[10pt,twoside,a4paper,reqno]{amsart}

\usepackage{amsmath,amsthm,amssymb,amsfonts}
\usepackage{fullpage}
\usepackage{hyperref}
\usepackage{enumerate}

\newtheorem{conjecture}{Conjecture}
\newtheorem{question}{Question}
\newtheorem{theorem}{Theorem}

\newtheorem{proposition}{Proposition}
\newtheorem{corollary}{Corollary}
\newtheorem{remark}{Remark}

\title{Transposition Cayley Spectra and Subgraph Counts}
\author{Boris Shapiro}
\address{Department of Mathematics, Stockholm University,
SE-106 91 Stockholm, Sweden} 
\email{shapiro@math.su.se}
\date{}

\begin{document}

\begin{abstract}
Let $G$ be a finite simple graph on the vertex set $[n]$, and let
$\Gamma(G)$ be the Cayley graph of the symmetric group $S_n$ generated
by the transpositions corresponding to the edges of $G$.  We study the
inverse spectral problem for the assignment $G\mapsto \Gamma(G)$ through
the group-algebra operator
$A_G=\sum_{\{i,j\}\in E(G)}(ij)$ and its regular spectral moments
$M_k(G)=\operatorname{tr}(A_G^k)$.  The moments admit two all-order
forms: a universal support expansion in non-induced subgraph counts of
$G$, and a representation-theoretic expansion over the irreducible
representations of $S_n$.  We compute the first nontrivial moments
through order six.  In particular, the fourth moment already contains
the Coxeter triangle relation and determines the signed local invariant
$6N_{K_3}(G)-N_{P_3}(G)$, rather than $N_{P_3}(G)$ alone. Consequently,
for triangle-free graphs the spectrum determines the second moment of
the degree sequence.  We also give a simplified sixth-moment formula for
triangle-free graphs, showing that in this class the spectrum determines
$5N_{K_{1,3}}(G)+2N_{P_4}(G)+12N_{C_4}(G)$.  Further structural
consequences include bipartiteness and spectral symmetry of $\Gamma(G)$,
a precise connectedness test, and the identification of the standard
irreducible component with the shifted graph Laplacian $mI-L_G$.

We add several new results.  (i) A Hurwitz-type rank bound: an
identity word of length $k$ can only be supported on a graph $F$ with
$k\ge 2(|V(F)|-c(F))$.  The support coefficients are therefore filtered
by the rank of the graphic matroid, and they are multiplicative under
disjoint union in the exponential sense.  (ii) In top degree the
coefficients add up to the genus-zero Hurwitz numbers
$(2n-2)!\,n^{n-3}$.  (iii) On the permutation modules
$M^{(n-k,k)}$ the operator $mI-A_G$ is the Laplacian of the $k$-token
graph of $G$, so the two-row components of the refined spectrum are
token-graph Laplacian spectra.
Together with the Caputo--Liggett--Richthammer theorem this shows that
$\operatorname{Spec}\Gamma(G)$ determines the algebraic connectivity of
$G$.  (iv) On the class of regular graphs the fourth moment determines the
number of triangles, and the sixth moment determines
$6N_{C_4}+11N_{K_4-e}+8N_{K_4}$; in particular, for triangle-free regular
graphs it determines the number of $4$-cycles.  (v) Before diagonal
specialization of the edge weights, the single weighted fourth regular
trace reconstructs the line graph and hence every connected simple graph.
(vi) We give the complete eighth moment, and we prove that
$M_2,M_4,M_6$ determine neither the number of triangles, nor the degree
sequence, nor connectedness.  (vii) We report computer-assisted
verification of the reconstruction conjecture.  The CayleyPy group
(I.~Koltsov) checked all $274\,668$ graphs on $9$ vertices.  We confirm this
independently, and show that already $M_2,\dots,M_{10}$ separate all
graphs on at most $9$ vertices, whereas $M_2,\dots,M_8$ do not, even for
the number of triangles.  The
paper is positioned as a theoretical spectral-moment complement to the
recent CayleyPy computational programme on Cayley graphs.
\end{abstract}

\maketitle

\section{Introduction}

Let $G$ be a finite simple graph with vertex set $[n]=\{1,\dots,n\}$.
Associate with every edge $\{i,j\}\in E(G)$ the transposition
$(ij)\in S_n$, and put
\[
T_G=\{(ij):\{i,j\}\in E(G)\}.
\]
We consider the Cayley graph
\[
\Gamma(G)=\operatorname{Cay}(S_n,T_G).
\]
Equivalently, $\Gamma(G)$ is the state graph of the interchange process
on $G$ before assigning transition probabilities.  Its adjacency
operator is left multiplication by
\[
A_G=\sum_{\{i,j\}\in E(G)}(ij)
\]
in the group algebra $\mathbb C[S_n]$.

\medskip
The paper is concerned with the following inverse spectral problem.

\begin{question}\label{q:main-inverse}
If two graphs $G_1$ and $G_2$ have cospectral transposition Cayley graphs
$\Gamma(G_1)$ and $\Gamma(G_2)$, must $G_1$ and $G_2$ be isomorphic?
More generally, which invariants of $G$ are determined by
$\operatorname{Spec}(\Gamma(G))$?
\end{question}

This question is motivated by several sources.  Ordinary graph spectra
are notoriously non-rigid; Godsil--McKay switching provides a systematic
source of cospectral non-isomorphic graphs \cite{GodsilMcKay}.  On the
other hand, the Cayley graph $\Gamma(G)$ remembers not only the adjacency
relations in $G$, but also the multiplication rules among the
corresponding transpositions in $S_n$.  Thus it is natural to ask whether
passing from $G$ to $\Gamma(G)$ produces a substantially stronger
spectral invariant.

The same family of operators also appears in the study of random walks
and the interchange process.  Aldous' spectral gap conjecture, proved by
Caputo, Liggett and Richthammer \cite{CLR}, asserts that the spectral gap
of the interchange process on an arbitrary graph equals the spectral gap
of the ordinary random walk on that graph.  Cesi studied related Cayley
graphs for complete multipartite transposition sets and obtained explicit
representation-theoretic eigenvalue formulae in that case
\cite{Cesi}.  The present paper focuses on a different aspect: the full
spectrum of $\Gamma(G)$ as a possible inverse invariant of the generating
graph $G$.

A recent computational direction which should be mentioned explicitly is
the CayleyPy project of Chervov and collaborators
\cite{ChervovCube,ChervovRL,ChervovGrowth,ChervovHolography}.  That
project develops efficient computational and machine-learning tools for
Cayley and Schreier graphs, especially for path-finding, diameter, growth,
and large-scale experiments.  In particular, \cite{ChervovRL} studies
Cayley graphs of $S_n$ generated by cyclic shifts and a transposition,
proposes diameter conjectures and spectral-distribution phenomena, while
\cite{ChervovGrowth} announces the public CayleyPy library and about two
hundred conjectures, mostly concerning diameters and growth of Cayley and
Schreier graphs.  The relation with the present note is complementary:
CayleyPy supplies a powerful experimental framework for Cayley graphs,
whereas here we isolate exact algebraic moment formulae for the special
family obtained from arbitrary transposition graphs.

There is also a closely related, but different, rigidity question in
which one asks for the automorphism group or the abstract isomorphism
type of a Cayley graph generated by transpositions.  Results of Feng,
Ganesan, and, in greatest generality for arbitrary transposition sets,
Gijswijt--de Meijer show that for $n\geq 5$ these Cayley graphs are
normal and that their automorphism groups are controlled by the
automorphism groups of the underlying transposition graphs
\cite{Feng,Ganesan,GijswijtDeMeijer}.  Those results concern graph
isomorphism and automorphisms, whereas the present note asks what can be
recovered from the adjacency spectrum alone.

A complementary representation-theoretic refinement is developed in the
author's paper \cite{ShapiroN22}, to appear in the \emph{Electronic
Journal of Combinatorics}.  There the first genuinely new irreducible
component, corresponding to the partition $(n-2,2)$, is studied in detail:
an explicit edge-space model and coordinate formula are given, its trace
moments are expanded in support-subgraph counts, and weighted trace
polynomials are used for reconstruction.  The two papers address different
levels of information.  Reference~\cite{ShapiroN22} keeps the
$(n-2,2)$ block labelled, whereas the present paper asks what survives after
all irreducible blocks are merged into the ordinary regular spectrum of
$\Gamma(G)$.  This distinction will be important below.

The novelty of the paper should therefore be understood as follows.  The
representation-theoretic decomposition of the regular representation and
the special spectra of the complete and star transposition graphs are
classical or already present in the literature.  The recent CayleyPy
papers have also made Cayley graph computation, growth, diameter and
path-finding problems highly visible and have explicitly raised spectral
questions for large families of Cayley graphs
\cite{ChervovRL,ChervovGrowth}.  What seems not to have been isolated in
this form is the inverse-spectral moment viewpoint for an arbitrary
transposition graph $G$, especially the universal support expansion and
the explicit low-order consequences for reconstructing graph invariants.

Our approach is based on moments.  Let
\[
M_k(G)=\operatorname{tr}(A_G^k)
\]
where the trace is taken in the regular representation of $S_n$.  Since
only the coefficient of the identity permutation contributes to the
regular trace, $M_k(G)/n!$ counts words of length $k$ in the
edge-transpositions of $G$ whose product is the identity.  This elementary
observation leads to a universal expansion
\[
\frac{M_k(G)}{n!}=\sum_F c_k(F)N_F(G),
\]
where $F$ ranges over finite graphs with no isolated vertices, $N_F(G)$
is the number of non-induced copies of $F$ in $G$, and the coefficients
$c_k(F)$ are universal integers depending only on $F$ and $k$.

The same moments also have a representation-theoretic form.  If
$V_\lambda$ runs through the irreducible representations of $S_n$ and
$f^\lambda=\dim V_\lambda$, then
\[
M_k(G)=\sum_{\lambda\vdash n} f^\lambda
\operatorname{tr}\left(\left(\sum_{e\in E(G)}\rho_\lambda(\tau_e)\right)^k\right).
\]
These two formulae organize the same information in two different ways:
the first by the supports of words in the generating graph, and the
second by the irreducible representations of the symmetric group.

The central object of the present paper is not the individual moments
themselves but rather the embedding
\[
G \longmapsto A_G=\sum_{\{i,j\}\in E(G)}(ij)\in \mathbb{C}[S_n].
\]
This construction associates to a graph a noncommutative operator in the
group algebra of the symmetric group. The moments
\[
M_k(G)=\operatorname{tr}(A_G^k)
\]
count closed transposition words reducing to the identity permutation.
The resulting support expansions are governed by the Coxeter relations
in $S_n$, namely the cancellation relations $(ij)^2=e$ and the braid
relations
\[
(ij)(jk)(ij)=(jk)(ij)(jk).
\]
In this sense, the Cayley spectrum induces a finite-type graph invariant
theory controlled by the Coxeter structure of the symmetric group.

The explicit calculations in Section~\ref{sec:first-moments} show that
$M_2$ recovers the number of edges and that $M_4$ already detects the
Coxeter triangle relation among the three transpositions attached to a
triangle in $G$.  More precisely, the fourth moment determines the signed
local invariant $6N_{K_3}(G)-N_{P_3}(G)$.  Hence, on the important
subclass of triangle-free graphs, it recovers the number of adjacent
edge-pairs, equivalently the second moment of the degree sequence.  We
also compute $M_6$ completely and derive a cleaner specialization for
triangle-free graphs, where it determines the next local combination
$5N_{K_{1,3}}(G)+2N_{P_4}(G)+12N_{C_4}(G)$.  These formulae give concrete
evidence that the transposition Cayley spectrum sees local structure of
$G$ in a way which differs sharply from the ordinary adjacency spectrum.

\medskip
\noindent\textbf{New in this version.}
\begin{enumerate}[(a)]
\item \emph{Rank filtration and exponential formula}
(Theorem~\ref{thm:rank-bound}, Proposition~\ref{prop:exp-formula}).  If
$c_k(F)\neq0$, then $k\ge 2\operatorname{rk}(F)$, where
$\operatorname{rk}(F)=|V(F)|-c(F)$ is the rank of the cycle matroid of
$F$.  The coefficients of a disjoint union are obtained from those of
its components by binomial convolution.  Only finitely many
\emph{connected} supports contribute to each moment, namely those with
at most $k/2+1$ vertices.
\item \emph{Hurwitz numbers} (Corollary~\ref{cor:hurwitz}).  In top
degree $k=2(n-1)$, summing the coefficients $c_k(F)$ over the spanning
connected subgraphs of $K_n$ gives the genus-zero Hurwitz number
$(2n-2)!\,n^{n-3}$.
\item \emph{Token graphs and the spectral gap}
(Theorem~\ref{thm:token}, Corollary~\ref{cor:aldous}).  For two-row
shapes, $m-\operatorname{Spec}A_G^{(n-k,k)}$ is the ``new part'' of the
Laplacian spectrum of the $k$-token graph $F_k(G)$.  By the
Caputo--Liggett--Richthammer theorem, the second largest eigenvalue of
$\Gamma(G)$ equals $m-a(G)$, where $a(G)$ is the algebraic connectivity.
Hence $\operatorname{Spec}\Gamma(G)$ determines $a(G)$.
\item \emph{Weighted fourth-trace reconstruction}
(Theorem~\ref{thm:weighted-fourth}).  If the edge-transpositions are given
independent variables, the coefficient of $y_e^2y_f^2$ in the fourth
regular trace distinguishes adjacent from disjoint pairs of edges.  Thus a
single weighted fourth trace reconstructs the line graph and, for connected
graphs, the graph itself.
\item \emph{Regular graphs} (Corollary~\ref{cor:regular-graphs}).  Within
the class of regular graphs, the fourth Cayley moment determines the number
of triangles, while the sixth determines
$6N_{C_4}+11N_{K_4-e}+8N_{K_4}$.  Hence the number of $4$-cycles is
spectrally determined for triangle-free regular graphs.
\item \emph{The eighth moment} (Theorem~\ref{thm:eighth}), together with
the connected part of $c_{10}$ for trees and cycles, and conjectural
closed formulae for stars, cycles and trees
(Conjecture~\ref{conj:closed}).
\item \emph{Limits of low moments} (Proposition~\ref{prop:counterexamples}).
There are pairs of graphs on $7$ vertices with equal $M_2,M_4,M_6$ but
different numbers of triangles, different degree sequences, or
different connectivity.  There are pairs on $9$ vertices with equal
$M_2,\dots,M_8$ but different numbers of triangles and different degree
sequences.
\item \emph{Computations} (Section~\ref{sec:computations}).  The
CayleyPy group computed the refined spectra of all graphs with
$n\le 9$ and found no collisions.  We confirm this independently: for
$n\le 9$ the moments $M_2,\dots,M_{10}$ alone already separate all
graphs.
\end{enumerate}

\section{The basic operator}

The central object is the element
\[
A_G=\sum_{\{i,j\}\in E(G)} (ij)
\in \mathbb{C}[S_n].
\]

The adjacency matrix of $\Gamma(G)$ is precisely the operator
of left multiplication by $A_G$ in the regular representation
of $S_n$.

Thus the spectrum of $\Gamma(G)$ coincides with the spectrum of
$A_G$ acting in the regular representation.

Since
\[
\mathbb{C}[S_n]
\simeq
\bigoplus_{\lambda\vdash n}
(\dim V_\lambda)\,V_\lambda,
\]
one obtains
\[
\operatorname{Spec}(\Gamma(G))
=
\bigcup_{\lambda\vdash n}
\operatorname{Spec}(\rho_\lambda(A_G)),
\]
where $\rho_\lambda$ denotes the irreducible representation
corresponding to the partition $\lambda$.

This representation-theoretic decomposition suggests that the spectrum
of $\Gamma(G)$ should contain much richer information than the ordinary
adjacency spectrum of $G$.

There is also a simple global feature which is useful throughout the
paper.

\begin{proposition}[Bipartite symmetry]
\label{prop:bipartite-symmetry}
For every graph $G$, the Cayley graph $\Gamma(G)$ is bipartite.  Its two
parts are the even and odd permutations in $S_n$.  Consequently its
adjacency spectrum is symmetric about zero, and
\[
 M_{2r+1}(G)=0
 \qquad (r\geq 0).
\]
\end{proposition}

\begin{proof}
Every generator in $T_G$ is a transposition, hence changes the parity of
a permutation.  Thus every edge of $\Gamma(G)$ joins an even permutation
to an odd permutation.  The spectral symmetry is the standard spectral
symmetry of a finite bipartite graph, and the vanishing of odd moments is
equivalent to the absence of closed walks of odd length.
\end{proof}

\section{All-order moment formulae}

The low-order calculations are only the first instances of a more
systematic structure.  There are two useful all-order formulae.  The
first is purely combinatorial and expresses every moment as a universal
linear combination of subgraph-counting functions.  The second is
representation-theoretic and expresses the same moments through the
irreducible representations of $S_n$.

\subsection{The support expansion}

A natural family of spectral invariants is given by the moments
\[
M_k(G)=\operatorname{tr}(A_G^k).
\]
Let $[1]B$ denote the coefficient of the identity permutation in an
element $B\in\mathbb C[S_n]$.  Since the trace of left multiplication
by a non-identity element of $S_n$ in the regular representation is zero,
whereas the trace of the identity is $n!$, one has
\[
M_k(G)=n!\,[1]A_G^k.
\]
Equivalently,
\[
\frac{M_k(G)}{n!}
=
\#\left\{(e_1,\ldots,e_k)\in E(G)^k:
\tau_{e_1}\cdots \tau_{e_k}=1\right\},
\]
where $\tau_{\{i,j\}}=(ij)$.  Thus the moments count words in the
edge-transpositions whose product is trivial.

Let $F$ be a finite simple graph with no isolated vertices.  Write
$E(F)$ for its edge set, and regard every edge $e=\{a,b\}$ of $F$ as the
transposition $(ab)$ in the symmetric group on $V(F)$.  For $k\geq 0$
define
\[
 c_k(F)=
 \#\left\{(e_1,\ldots,e_k)\in E(F)^k:
 \begin{array}{l}
   e_1,\ldots,e_k \text{ use every edge of }F,\\[2mm]
   \tau_{e_1}\cdots \tau_{e_k}=1
 \end{array}
 \right\}.
\]
Thus $c_k(F)$ is a universal integer depending only on the isomorphism
class of $F$.

Let $N_F(G)$ denote the number of subgraphs of $G$ which are isomorphic
to $F$, not necessarily induced.  Then the normalized moment has the
following universal expansion.

\begin{theorem}[Universal support expansion]
For every simple graph $G$ and every $k\geq 0$ one has
\[
 \frac{M_k(G)}{n!}
 =
 \sum_F c_k(F) N_F(G),
\]
where the sum is over all isomorphism classes of finite simple graphs
$F$ with no isolated vertices and with at most $k$ edges.
\end{theorem}

\begin{proof}
By the coefficient-of-the-identity formula,
\[
 \frac{M_k(G)}{n!}
 =
 \#\{(e_1,\ldots,e_k)\in E(G)^k:
      \tau_{e_1}\cdots\tau_{e_k}=1\}.
\]
To each such word associate its support, namely the subgraph of $G$
spanned by the distinct edges appearing among $e_1,\ldots,e_k$ and their
endpoints.  The support has no isolated vertices and at most $k$ edges.
For a fixed subgraph $H\subseteq G$ isomorphic to $F$, the number of
identity words of length $k$ whose support is exactly $H$ is precisely
$c_k(F)$, since the transposition relations depend only on the abstract
incidence pattern of $H$.  Summing over all possible supports gives the
formula.
\end{proof}

This theorem is perhaps the cleanest general form of the moment method.
It says that the Cayley spectrum determines, for every $k$, one explicit
linear combination of small subgraph counts.  The problem of extracting
particular graph invariants from the spectrum becomes the problem of
inverting, or partially triangularizing, these universal linear
combinations.

\begin{remark}
The constants $c_k(F)$ are algorithmically computable.  For fixed $k$
there are only finitely many graphs $F$ with at most $k$ edges, and one
can enumerate all words of length $k$ in the edge-transpositions of $F$.
This gives a finite, completely explicit procedure for producing the
moment formula of any prescribed order.
\end{remark}

\begin{remark}
The support expansion gives another way to see the parity phenomenon in
Proposition~\ref{prop:bipartite-symmetry}: $c_k(F)=0$ for every odd
$k$.  The first nontrivial moments are therefore $M_2,M_4,M_6,\ldots$.
\end{remark}

\subsection{The rank filtration}

For a graph $F$ write $c(F)$ for its number of connected components and
\[
 \operatorname{rk}(F)=|V(F)|-c(F)
\]
for the rank of its cycle matroid (the size of a spanning forest).

\begin{theorem}[Hurwitz-type rank bound]
\label{thm:rank-bound}
If $c_k(F)\neq 0$, then
\[
 2\operatorname{rk}(F)\le k
 \qquad\text{and}\qquad |E(F)|\le k .
\]
In particular a connected support of an identity word of length $k$ has
at most $k/2+1$ vertices.  Consequently
\[
 \frac{M_{2r}(G)}{n!}=
 \sum_{\substack{F:\ \operatorname{rk}(F)\le r\\ |E(F)|\le 2r}}
 c_{2r}(F)\,N_F(G).
\]
\end{theorem}

\begin{proof}
The inequality $|E(F)|\le k$ is clear, since a word of length $k$ uses
at most $k$ distinct letters.  Let $(e_1,\dots,e_k)$ be a word counted
by $c_k(F)$.  Put $\sigma_0=1$ and $\sigma_j=\tau_{e_j}\sigma_{j-1}$, and
let $\operatorname{cyc}(\sigma)$ be the number of cycles of $\sigma$ on
$V(F)$, fixed points included.  Multiplying by a transposition $(ab)$
changes the number of cycles by $\pm1$.  It decreases by one exactly
when $a$ and $b$ lie in different cycles of the permutation being
multiplied.

Let $F_j$ be the graph on $V(F)$ with edges $e_1,\dots,e_j$.  The cycles
of $\sigma_{j-1}$ are contained in the orbits of the group generated by
$\tau_{e_1},\dots,\tau_{e_{j-1}}$, and these orbits are the connected
components of $F_{j-1}$.  Hence, whenever $e_j$ joins two different
components of $F_{j-1}$, the number of cycles drops by one at step $j$.
Passing from the edgeless graph $F_0$ to $F_k=F$ requires exactly
$|V(F)|-c(F)=\operatorname{rk}(F)$ such merging steps.  So at least
$\operatorname{rk}(F)$ steps decrease the number of cycles.  Since
$\sigma_k=\sigma_0=1$, the numbers of decreasing and increasing steps
are equal.  Therefore $k\ge 2\operatorname{rk}(F)$.
\end{proof}

\begin{remark}
For a connected support $F$ on $v$ vertices the bound reads
$k\ge 2(v-1)$.  By the Riemann--Hurwitz formula, an identity word of
length $k$ generating a transitive subgroup corresponds to a branched
covering of the sphere by a surface of genus $g$ with
$k=2(v-1)+2g$.  The words with $k=2(v-1)$ are the \emph{minimal
transitive factorizations} of the identity.  Their total number in $S_n$
is a classical Hurwitz number; see Corollary~\ref{cor:hurwitz}.  Thus
the support expansion is naturally graded by genus, and the support
graphs $F$ refine the Hurwitz count according to which transpositions
are used.
\end{remark}

\begin{proposition}[Exponential formula]
\label{prop:exp-formula}
For a graph $F$ put
\[
 C_F(x)=\sum_{k\ge0}c_k(F)\frac{x^k}{k!}.
\]
If $F=F_1\sqcup F_2$ is a vertex-disjoint union, then
$C_F(x)=C_{F_1}(x)\,C_{F_2}(x)$, that is,
\[
 c_k(F_1\sqcup F_2)=\sum_{j=0}^k\binom kj c_j(F_1)\,c_{k-j}(F_2).
\]
Moreover
\[
 \sum_{k\ge0}\frac{M_k(G)}{n!}\frac{x^k}{k!}
 =[1]\,e^{xA_G}
 =\sum_{F}C_F(x)\,N_F(G),
\]
where $F$ runs over all graphs without isolated vertices, the empty
graph included, with $C_\emptyset=1$.
\end{proposition}

\begin{proof}
A word in $E(F_1)\sqcup E(F_2)$ is a shuffle of a word in $E(F_1)$ and a
word in $E(F_2)$, and there are $\binom kj$ ways to interleave words of
lengths $j$ and $k-j$.  The transpositions of $F_1$ commute with those
of $F_2$, and the corresponding subgroups intersect trivially.  So the
product is the identity if and only if both subwords have identity
product.  A word uses every edge of $F$ if and only if each subword uses
every edge of its own component.  The second statement is the
exponential generating function form of the support expansion.
\end{proof}

By Theorem~\ref{thm:rank-bound} and Proposition~\ref{prop:exp-formula},
the moment $M_k$ is determined by finitely many numbers $c_j(F)$, $j\le
k$, for \emph{connected} graphs $F$ with at most $k/2+1$ vertices.  This
makes computing higher moments very cheap: $M_8$ needs only the
connected graphs on at most five vertices, and $M_{10}$ only those on at
most six vertices.

\begin{corollary}[Genus-zero Hurwitz numbers]
\label{cor:hurwitz}
For every $n\ge 2$,
\[
 \sum_{\substack{F\subseteq K_n\\ F\ \text{spanning, connected}}}
 c_{2n-2}(F)=(2n-2)!\,n^{n-3},
\]
where the sum is over labelled spanning connected subgraphs of $K_n$.
Equivalently,
\[
\sum_{[F]}c_{2n-2}(F)\,\frac{n!}{|\operatorname{Aut}F|}=(2n-2)!\,n^{n-3},
\]
the sum running over isomorphism classes of connected graphs on $n$
vertices.
\end{corollary}

\begin{proof}
The left-hand side counts words of length $2n-2$ in the transpositions
of $S_n$ whose product is $1$ and which generate a transitive subgroup,
i.e.\ minimal transitive factorizations of the identity.  By Hurwitz's
formula (proved rigorously by Goulden and Jackson
\cite{Hurwitz,GouldenJackson97}), the number of minimal transitive
factorizations of a permutation of cycle type
$\alpha=(\alpha_1,\dots,\alpha_\ell)\vdash n$ is
$(n+\ell-2)!\,n^{\ell-3}\prod_i\alpha_i^{\alpha_i}/(\alpha_i-1)!$.
For $\alpha=1^n$ this equals $(2n-2)!\,n^{n-3}$.
\end{proof}

We used Corollary~\ref{cor:hurwitz} as an independent check of the
coefficient tables in Section~\ref{sec:first-moments}.  For instance,
for $n=4$ it reads
\[
4\cdot 30+12\cdot 42+12\cdot 96+3\cdot 72+6\cdot 132+1\cdot 96=2880=6!\cdot 4,
\]
where the terms correspond to $K_{1,3},P_4,\operatorname{Paw},C_4,K_4-e,K_4$.
We checked the identity by computer for all $n\le 6$.

\subsection{Representation-theoretic form}

The preceding formula is combinatorial.  There is also an all-order
formula directly in the representation theory of the symmetric group.
Let $f^\lambda=\dim V_\lambda$, and let
\[
 A_G^\lambda=\rho_\lambda(A_G)
 =\sum_{\{i,j\}\in E(G)}\rho_\lambda((ij))
 \in \operatorname{End}(V_\lambda).
\]
Since the regular representation decomposes as
\[
 \mathbb C[S_n]\simeq \bigoplus_{\lambda\vdash n}
       (f^\lambda) V_\lambda,
\]
one obtains, for every $k\geq 0$,
\[
 M_k(G)=
 \sum_{\lambda\vdash n} f^\lambda
 \operatorname{tr}\left((A_G^\lambda)^k\right).
\]
Equivalently,
\[
 [1]A_G^k
 =
 \frac1{n!}
 \sum_{\lambda\vdash n} f^\lambda
 \operatorname{tr}\left((A_G^\lambda)^k\right).
\]
This is the representation-theoretic all-order moment formula.

It is important, however, that $A_G$ is not central unless $G$ is the
complete graph or the empty graph.  Therefore the formula cannot usually
be reduced to ordinary irreducible characters alone.  One needs the full
matrices of the transpositions in irreducible representations.

A particularly concrete version is obtained by using the Young
orthogonal form.  If $\operatorname{SYT}(\lambda)$ is the set of standard
Young tableaux of shape $\lambda$, then
\[
\begin{aligned}
 &\operatorname{tr}\left((A_G^\lambda)^k\right) \\
 &\quad =
 \sum_{T_0,T_1,\ldots,T_{k-1}\in \operatorname{SYT}(\lambda)}
 \sum_{e_1,\ldots,e_k\in E(G)}
 \rho_\lambda(\tau_{e_1})_{T_0,T_1}
 \rho_\lambda(\tau_{e_2})_{T_1,T_2}
 \cdots
 \rho_\lambda(\tau_{e_k})_{T_{k-1},T_0}.
\end{aligned}
\]
Thus the moments can be computed by summing closed walks in the basis of
standard tableaux, weighted by matrix coefficients of transpositions.
This gives a direct analogue of moment formulae for ordinary graph
spectra, but now with the vertices of the auxiliary walk replaced by
standard Young tableaux.

\begin{proposition}[Two all-order descriptions]
For every $k$, the same spectral invariant $M_k(G)$ admits both
expressions
\[
 \frac{M_k(G)}{n!}=\sum_F c_k(F)N_F(G)
\]
and
\[
 M_k(G)=
 \sum_{\lambda\vdash n} f^\lambda
 \operatorname{tr}\left(
 \left(\sum_{e\in E(G)}\rho_\lambda(\tau_e)\right)^k
 \right).
\]
The first expression organizes the answer by subgraph supports in $G$;
the second organizes it by irreducible representations of $S_n$.
\end{proposition}

\subsection{The standard component and the Laplacian of $G$}

There is one irreducible component of the preceding representation-
theoretic formula which has an especially transparent graph-theoretic
meaning.  Let $W=\mathbb C^n$ be the permutation representation of
$S_n$, and let $\mathbf 1\subset W$ be the one-dimensional trivial
subrepresentation.  Its complement
\[
 W_0=\{x=(x_1,\ldots,x_n): x_1+\cdots+x_n=0\}
\]
is the standard irreducible representation of $S_n$, corresponding to
the partition $(n-1,1)$.

Let $L_G$ be the usual graph Laplacian of $G$, and let
$m=|E(G)|$.  On $W$, the transposition $(ij)$ acts by interchanging the
coordinates $x_i$ and $x_j$.  Hence
\[
 (ij)=I-(e_i-e_j)(e_i-e_j)^t
\]
with respect to the standard coordinate basis.  Summing over the edges of
$G$ gives the following elementary but useful identity.

\begin{proposition}[The standard component]
\label{prop:standard-laplacian}
On the permutation representation $W=\mathbb C^n$ one has
\[
 \sum_{\{i,j\}\in E(G)}(ij)=mI-L_G.
\]
Consequently, on the standard representation $W_0$ the operator
$A_G^{(n-1,1)}$ has eigenvalues
\[
 m-\lambda_2,\,m-\lambda_3,\ldots,\,m-\lambda_n,
\]
where
\[
 0=\lambda_1\leq \lambda_2\leq\cdots\leq\lambda_n
\]
are the Laplacian eigenvalues of $G$.
\end{proposition}

\begin{proof}
For a single edge $\{i,j\}$, the transposition $(ij)$ fixes the
orthogonal complement of the line spanned by $e_i-e_j$ and sends
$e_i-e_j$ to its negative.  Therefore, as an operator on $W$,
\[
 (ij)=I-(e_i-e_j)(e_i-e_j)^t.
\]
Summing this identity over all edges gives
\[
 A_G=mI-\sum_{\{i,j\}\in E(G)}(e_i-e_j)(e_i-e_j)^t=mI-L_G.
\]
The vector $(1,\ldots,1)$ is an eigenvector with eigenvalue $m$, and the
restriction to its complement is exactly the standard representation.
This proves the assertion.
\end{proof}

\begin{remark}
Proposition~\ref{prop:standard-laplacian} should be interpreted with a
minor caveat.  The full Cayley spectrum records the union of the spectra
of all irreducible components, with multiplicities, but it does not label
which eigenvalues come from which component.  Thus the proposition shows
that the \emph{refined} irreducible spectrum contains the Laplacian
spectrum of $G$.  It also explains why the transposition Cayley spectrum
is naturally expected to be stronger than the ordinary adjacency or
Laplacian spectrum: the latter appears already in the single component
$(n-1,1)$, while all other irreducible representations provide additional
constraints.
\end{remark}

\subsection{Two-row components and token graphs}

Proposition~\ref{prop:standard-laplacian} is the case $k=1$ of a general
statement about permutation modules.  For $0\le k\le n$, the
\emph{$k$-token graph} $F_k(G)$ has as vertices the $k$-element subsets
of $[n]$.  Two subsets $S,S'$ are adjacent if $S\triangle S'=\{i,j\}$
with $\{i,j\}\in E(G)$ \cite{FabilaMonroy,DalfoEtAl}.  Let
$M^{(n-k,k)}=\mathbb C\binom{[n]}{k}$ be the corresponding permutation
module.  For $k\le n/2$, Young's rule gives
\[
 M^{(n-k,k)}\simeq\bigoplus_{j=0}^{k}V_{(n-j,j)} .
\]

\begin{theorem}[Token graphs]
\label{thm:token}
On $M^{(n-k,k)}$ one has
\[
 mI-\sum_{\{i,j\}\in E(G)}(ij)=L_{F_k(G)},
\]
the Laplacian of the $k$-token graph.  Consequently, for $1\le k\le n/2$,
the multiset $m-\operatorname{Spec}A_G^{(n-k,k)}$ is the multiset
difference
\[
 \operatorname{LSpec}F_k(G)\setminus\operatorname{LSpec}F_{k-1}(G),
\]
where $\operatorname{LSpec}$ denotes the Laplacian spectrum.
\end{theorem}

\begin{proof}
Let $S$ be a $k$-subset and $e=\{i,j\}\in E(G)$.  If $e\subseteq S$ or
$e\cap S=\emptyset$, then $(ij)S=S$.  Otherwise
$(ij)S=S\triangle\{i,j\}$, which is a neighbour of $S$ in $F_k(G)$, and
every neighbour arises from exactly one edge.  Hence $A_G$ acts on
$M^{(n-k,k)}$ as $A(F_k(G))+D$, where $D$ is diagonal with entry
$m-\deg_{F_k(G)}(S)$ at $S$.  Thus $mI-A_G=L_{F_k(G)}$.  The second
statement follows by induction on $k$ from Young's rule.
\end{proof}

The case $k=1$ is Proposition~\ref{prop:standard-laplacian}.  The
inclusion $\operatorname{LSpec}(G)\subseteq\operatorname{LSpec}F_k(G)$
proved in \cite{DalfoEtAl} is an immediate consequence of
Theorem~\ref{thm:token}, since $V_{(n-1,1)}\subseteq M^{(n-k,k)}$.  For
$\mu=(1^n)$ one recovers $\Gamma(G)$ itself, which is the $n$-coloured
token graph.  For a general composition $\mu$, the Laplacian of the
Schreier graph of $S_n/S_\mu$ with respect to $T_G$ is $mI-A_G$ acting
on $M^\mu$.

For $k=2$, Theorem~\ref{thm:token} isolates the first irreducible
correction to the ordinary Laplacian spectrum:
\[
 m-\operatorname{Spec}A_G^{(n-2,2)}
 =\operatorname{LSpec}F_2(G)\setminus\operatorname{LSpec}G
\]
as a multiset difference.  This component is the subject of
\cite{ShapiroN22}.  Among other results, that paper gives an explicit
edge-space model for $A_G^{(n-2,2)}$, all-order support-subgraph trace
formulae, and a complete cubic inversion for trees after adjoining the
ordinary Laplacian spectrum.  Thus the token-graph description above and
the support expansion of the present paper meet already in the first new
two-row component.

\begin{corollary}[The spectrum determines the algebraic connectivity]
\label{cor:aldous}
Let $G$ be connected, and let $a(G)=\lambda_2(L_G)$ be its algebraic
connectivity.  Then the largest eigenvalue of $\Gamma(G)$ other than $m$
is
\[
 \lambda_2(\Gamma(G))=m-a(G),
\]
and by bipartite symmetry the smallest eigenvalue other than $-m$ is
$-(m-a(G))$.  In particular $\operatorname{Spec}\Gamma(G)$ determines
$a(G)$.  Also, for every Laplacian eigenvalue $\mu\neq 0$ of $G$, the
number $m-\mu$ is an eigenvalue of $\Gamma(G)$ with multiplicity at
least $(n-1)\operatorname{mult}_{L_G}(\mu)$.
\end{corollary}

\begin{proof}
The operator $mI-A_G$ on $\mathbb C[S_n]$ is the generator of the
interchange process on $G$ with unit rates, and $L_G$ is the generator
of the random walk on $G$.  The Caputo--Liggett--Richthammer theorem
\cite{CLR} (Aldous' spectral gap conjecture) says that their spectral
gaps coincide.  Since $G$ is connected, the eigenvalue $m$ of $\Gamma(G)$
is simple (Theorem~\ref{thm:connectedness-test}), so the spectral gap of
the interchange process is $m-\lambda_2(\Gamma(G))$.  This proves the
first statement.  The multiplicity bound follows from
Proposition~\ref{prop:standard-laplacian}, since $V_{(n-1,1)}$ occurs
$n-1$ times in the regular representation.
\end{proof}

\begin{remark}
In the language of token graphs, Corollary~\ref{cor:aldous} says that
$a(F_k(G))=a(G)$ for all $k$.  Dalf\'o et al.\ conjectured this, and it
was later recognised to be a consequence of the
Caputo--Liggett--Richthammer theorem \cite{DalfoEtAl,DalfoFiol}.  We
checked $\lambda_2(\Gamma(G))=m-a(G)$ numerically for all
connected graphs with at most $8$ vertices.
\end{remark}

\subsection{Finite-type consequences of the moment expansion}

The universal support expansion has the following immediate consequence.

\begin{corollary}[Moments are finite-type graph invariants]
For every fixed $k$, the normalized moment $M_k(G)/n!$ belongs to the
finite-dimensional span of the subgraph-counting functions $N_F(G)$ with
at most $k$ edges.  In particular, for fixed $k$ it depends only on the
numbers of small subgraphs of $G$.
\end{corollary}

Thus the inverse problem can be refined as follows.  For each $r\geq 1$,
one should study the linear span of
\[
 G\longmapsto \frac{M_{2}(G)}{n!},
 \frac{M_{4}(G)}{n!},\ldots,
 \frac{M_{2r}(G)}{n!}
\]
inside the vector space spanned by subgraph-counting functions
$N_F(G)$ with at most $2r$ edges.  The concrete question is which
individual subgraph counts can be recovered from these spectral
linear combinations.

The formula for $M_4$ below shows that the edge count and the signed
combination $6N_{K_3}-N_{P_3}$ are recoverable.  Thus the fourth moment
already sees the Coxeter relation among the three transpositions attached
to a triangle.  The formula for $M_6$ shows explicitly how the next local
support types enter the moment expansion.  This leads to the following
natural strengthening of the basic reconstruction problem.

\begin{question}
Do the functions $M_2,M_4,\ldots,M_{2r}$ determine any natural triangular
system of subgraph-counting functions $N_F(G)$, at least after restricting
to graphs with a fixed number of vertices or to natural subclasses such
as triangle-free graphs?
\end{question}

A particularly accessible subfamily consists of star supports.  Since
star counts are polynomial expressions in the degrees of $G$, this
suggests the following concrete problem.

\begin{question}
Do the even Cayley moments determine the power sums
\[
 \sum_{v\in V(G)} d_v^q
\]
of the degree sequence?  Equivalently, do they determine the degree
sequence of $G$?
\end{question}

The first case $q=1$ is $2|E(G)|$ and is determined by $M_2$.  The case
$q=2$ is determined by $M_4$ on triangle-free graphs, but in general the
triangle term in the fourth moment creates an obstruction which must be
separated by higher moments or additional representation-theoretic data.

\subsection{A weighted fourth trace and exact reconstruction}

The diagonal specialization in the ordinary moments hides which pairs of
edge-transpositions interact.  It is useful to record how much information
is present before that specialization.  Give the edges independent
commuting variables $\mathbf y=(y_e)_{e\in E(G)}$ and put
\[
 A_G(\mathbf y)=\sum_{e\in E(G)}y_e\tau_e,\qquad
 \mathcal R_k(G;\mathbf y)
 =\operatorname{tr}_{\mathrm{reg}}\bigl(A_G(\mathbf y)^k\bigr).
\]
Then $\mathcal R_k(G;\mathbf 1)=M_k(G)$, and
\[
 \mathcal R_k(G;\mathbf y)=n!\,[1]A_G(\mathbf y)^k.
\]
Thus every coefficient counts identity words with a prescribed edge
multiplicity, multiplied by $n!$.

\begin{theorem}[Weighted fourth-trace reconstruction]
\label{thm:weighted-fourth}
Let $G$ be a finite simple graph on $n$ vertices.  For every edge $e$,
\[
 [y_e^4]\,\mathcal R_4(G;\mathbf y)=n!,
\]
and for two distinct edges $e,f$,
\[
 [y_e^2y_f^2]\,\mathcal R_4(G;\mathbf y)=
 \begin{cases}
 4n!,& e\cap f\neq\varnothing,\\
 6n!,& e\cap f=\varnothing.
 \end{cases}
\]
Consequently the polynomial $\mathcal R_4(G;\mathbf y)$, considered up
to permutation of the edge variables, determines the line graph $L(G)$.
Moreover, among connected simple graphs it determines $G$ up to
isomorphism.
\end{theorem}

\begin{proof}
The coefficient of $y_e^4$ comes from the single word $eeee$, whose product
is the identity.  For $e\neq f$ there are six words containing two copies
of each letter.  If $e$ and $f$ are disjoint, the corresponding
transpositions commute, so all six words have identity product.  If they
are adjacent, four of the six words reduce by adjacent cancellations,
whereas the two alternating words $efef$ and $fefe$ are nontrivial because
$ef$ has order three.  This proves the coefficient formulae.

Hence the coefficient attached to each unordered pair of variables tells
whether the corresponding edges of $G$ meet, which is precisely the
adjacency relation in $L(G)$.  The pure fourth coefficient also recovers
$n!$, and hence $n$, whenever $G$ has an edge.  Whitney's line-graph theorem
\cite{Whitney} then reconstructs every connected simple graph, with the
only abstract line-graph ambiguity $K_3$ versus $K_{1,3}$ resolved by the
recovered value of $n$.  The connected edgeless case is just $K_1$.
\end{proof}

\begin{remark}
The theorem concerns a weighted refinement, not the ordinary Cayley
spectrum.  Setting every $y_e=1$ adds all pair coefficients together and
forgets which edge pairs are adjacent.  This is exactly why the unweighted
reconstruction problem remains nontrivial.  A parallel phenomenon occurs
in the $(n-2,2)$ component: \cite{ShapiroN22} shows that its weighted
quadratic trace already recovers $L(G)$ for $n\neq7$, while its weighted
quartic trace does so for every $n\ge4$.  Thus the same line-graph
mechanism appears both in the regular representation and in the first new
irreducible component, but at different trace degrees.
\end{remark}

\section{First moment calculations}\label{sec:first-moments}

We now record the first moments in the support-expansion language.  The
calculations below are meant to serve as the first concrete data for the
project.

We use the following notation.  The graph $P_r$ denotes the path on $r$
vertices, $C_r$ the cycle on $r$ vertices, $K_r$ the complete graph on
$r$ vertices, $K_{1,3}$ the claw, and $2K_2$ and $3K_2$ disjoint unions
of two and three edges, respectively.  We write $K_3\sqcup K_2$ and
$P_3\sqcup K_2$ for disjoint unions.  The graph $\operatorname{Paw}$ is
the triangle with one pendant edge, and $K_4-e$ is the complete graph on
four vertices with one edge removed.

\begin{theorem}[Moments up to order six]\label{thm:first-moments}
Let $G$ be a simple graph.  Then
\[
\frac{M_0(G)}{n!}=1,
\qquad
\frac{M_1(G)}{n!}=0,
\qquad
\frac{M_2(G)}{n!}=N_{K_2}(G),
\qquad
\frac{M_3(G)}{n!}=0,
\]
\[
\frac{M_4(G)}{n!}
=
N_{K_2}(G)+4N_{P_3}(G)+6N_{2K_2}(G)+12N_{K_3}(G),
\]
and
\[
\begin{aligned}
\frac{M_6(G)}{n!}
={}&N_{K_2}(G)
+20N_{P_3}(G)
+30N_{2K_2}(G)
+180N_{K_3}(G) \\
&+30N_{K_{1,3}}(G)
+42N_{P_4}(G)
+60N_{P_3\sqcup K_2}(G)
+90N_{3K_2}(G) \\
&+72N_{C_4}(G)
+96N_{\operatorname{Paw}}(G)
+180N_{K_3\sqcup K_2}(G) \\
&+132N_{K_4-e}(G)
+96N_{K_4}(G).
\end{aligned}
\]
Here all subgraph counts are non-induced.
\end{theorem}

\begin{proof}
The identities for $M_0$ and $M_1$ are immediate.  For $M_2$, a product
of two edge-transpositions is the identity if and only if the two
transpositions are equal.  Hence $M_2(G)/n!=N_{K_2}(G)$.

All odd moments vanish because a product of an odd number of
transpositions is an odd permutation and hence cannot be the identity.
This gives $M_3(G)=0$.

For $M_4$, the support of an identity word may have one, two, or three
edges.  If the support is a single edge, there is one word.  If the
support is a path $P_3$, i.e. two adjacent edges, there are four identity
words:
\[
aabb,
\quad abba,
\quad baab,
\quad bbaa.
\]
If the support is $2K_2$, i.e. two disjoint edges, the corresponding
transpositions commute, and all six words with two occurrences of each
edge are identity words.  Finally, if the support is a triangle with edge
transpositions $a,b,c$, then the Coxeter relation $aba=c$ and its cyclic
variants give twelve identity words of length four using all three
edges.  Equivalently, for every ordering of two adjacent triangle edges
one has a word of the form $abac=1$, and the same words are counted with
all cyclic choices of the distinguished edge.  This gives the displayed
fourth-moment formula.

The sixth moment is obtained from the universal support expansion by
enumerating the possible supports of length-six identity words.  There
are exactly thirteen support types with nonzero coefficient.  Their
coefficients $c_6(F)$ are as follows:
\[
\begin{array}{c|ccccccccccccc}
F & K_2 & P_3 & 2K_2 & K_3 & K_{1,3} & P_4 & P_3\sqcup K_2
& 3K_2 & C_4 & \operatorname{Paw} & K_3\sqcup K_2 & K_4-e & K_4 \\
\hline
c_6(F) & 1 & 20 & 30 & 180 & 30 & 42 & 60
& 90 & 72 & 96 & 180 & 132 & 96.
\end{array}
\]
Substituting these coefficients into
\[
\frac{M_6(G)}{n!}=\sum_F c_6(F)N_F(G)
\]
gives the formula.
\end{proof}

\begin{remark}[Check of the enumeration]
The table in the proof was obtained by the following finite procedure.
For each support graph $F$ with at most six edges one enumerates all
$6$-tuples of edges of $F$ which use every edge of $F$ and multiplies the
corresponding transpositions in $S_{V(F)}$.  The displayed coefficient is
the number of such words with product equal to the identity.  This also
provides a short independent way to reproduce the formula and to compute
higher moments.  In the present version, all coefficients were
recomputed by machine.  We computed $c_k(F)$ for connected $F$ by
inclusion--exclusion over edge subsets, $c_k(F)=\sum_{H\subseteq
E(F)}(-1)^{|E(F)\setminus H|}W_k(H)$, where $W_k(H)$ is the number of
identity words of length $k$ over $H$, obtained by dynamic programming on
$S_{V(F)}$.  The disconnected coefficients then follow from
Proposition~\ref{prop:exp-formula}.  For example
$c_6(P_3\sqcup K_2)=\binom64\cdot4\cdot1=60$ and
$c_6(3K_2)=6!/2^3=90$.  All entries of the table were confirmed.  We
also checked $M_k(G)/n!=W_k(E(G))$ directly against the full regular
spectrum for random graphs with $n\le5$.
\end{remark}

\begin{corollary}[The fourth-moment invariant]
\label{cor:fourth-moment-invariant}
The Cayley spectrum of $\Gamma(G)$ determines
\[
 |E(G)|
 \qquad\text{and}\qquad
 6N_{K_3}(G)-N_{P_3}(G).
\]
In particular, on the class of triangle-free graphs it determines
$N_{P_3}(G)$ and therefore
\[
 \sum_{v\in V(G)}d_v^2.
\]
\end{corollary}

\begin{proof}
Let $m=|E(G)|$, $P=N_{P_3}(G)$, $D=N_{2K_2}(G)$, and
$T=N_{K_3}(G)$.  The number $m$ is determined by $M_2(G)$.  Also
\[
D=\binom m2-P,
\]
because every unordered pair of distinct edges is either adjacent, hence
contributes to $P$, or disjoint, hence contributes to $D$.  Substituting
this into the fourth-moment formula gives
\[
 \frac{M_4(G)}{n!}
 =m+6\binom m2-2P+12T.
\]
Thus $6T-P$ is determined by the spectrum.  If $G$ is triangle-free, then
$T=0$, so $P$ is determined.  Finally,
\[
 P=\sum_{v\in V(G)}\binom{d_v}{2}
 \qquad\text{and}\qquad
 \sum_v d_v^2=2P+2m.
\]
\end{proof}

\begin{corollary}[Triangle-free sixth-moment consequence]
\label{cor:triangle-free-sixth}
Let $G$ be triangle-free.  Then the Cayley spectrum of $\Gamma(G)$
determines the linear combination
\[
 5N_{K_{1,3}}(G)+2N_{P_4}(G)+12N_{C_4}(G).
\]
If $G$ has girth at least five, this reduces to the determination of
\[
 5N_{K_{1,3}}(G)+2N_{P_4}(G).
\]
If, in addition, the maximum degree of $G$ is at most two, then the
spectrum determines $N_{P_4}(G)$.
\end{corollary}

\begin{proof}
Assume that $G$ is triangle-free and write
\[
 m=N_{K_2}(G),\quad P=N_{P_3}(G),\quad
 S=N_{K_{1,3}}(G),\quad R=N_{P_4}(G),\quad C=N_{C_4}(G).
\]
The disconnected support counts appearing in the sixth moment can be
eliminated in terms of these quantities.  First
\[
 N_{2K_2}=\binom m2-P.
\]
Next,
\[
 N_{P_3\sqcup K_2}=(m-2)P-3S-2R.
\]
Indeed, for each two-edge path $u-v-w$, the number of edges disjoint
from its three vertices is $m-d_u-d_v-d_w+2$; summing this over all
$P_3$'s and using triangle-freeness gives
\[
 \sum_{u-v-w}(d_u+d_v+d_w)=3S+4P+2R.
\]
Finally, classifying unordered triples of edges by their support gives
\[
 N_{3K_2}=\binom m3-N_{P_3\sqcup K_2}-S-R.
\]
Substitution into Theorem~\ref{thm:first-moments}, with all triangle
support types removed, yields
\[
\begin{aligned}
\frac{M_6(G)}{n!}
={}&m+30\binom m2+90\binom m3+(50-30m)P  \\
&+30S+12R+72C.
\end{aligned}
\]
By Corollary~\ref{cor:fourth-moment-invariant}, $m$ and $P$ are already
determined on the triangle-free class.  Hence $30S+12R+72C$, equivalently
$5S+2R+12C$, is determined.  The two stated specializations follow by
putting $C=0$ and then $S=0$.
\end{proof}

\begin{corollary}[Regular graphs: triangles and four-vertex supports]
\label{cor:regular-graphs}
Let $G$ be a $d$-regular graph on $n$ vertices.  Then the Cayley spectrum
of $\Gamma(G)$ determines the number of triangles $N_{K_3}(G)$.  It also
determines the combination
\[
 6N_{C_4}(G)+11N_{K_4-e}(G)+8N_{K_4}(G).
\]
In particular, if $G$ is triangle-free and regular, then the Cayley
spectrum determines $N_{C_4}(G)$.  Hence this holds, in particular, for
regular bipartite graphs.
\end{corollary}

\begin{proof}
Write $m=nd/2$ and $T=N_{K_3}(G)$.  Since
\[
 N_{P_3}(G)=n\binom d2,
\]
Corollary~\ref{cor:fourth-moment-invariant} immediately determines $T$;
more explicitly,
\[
 T=\frac1{12}\left(
 \frac{M_4(G)}{n!}-m-6\binom m2+2n\binom d2
 \right).
\]
We now use the sixth-moment formula.  In a regular graph the remaining
three-edge support counts are determined by $n,d,T$.  Indeed, putting
\[
 P=n\binom d2,\qquad S=n\binom d3,
\]
one has
\[
 N_{P_4}(G)=m(d-1)^2-3T,
\]
because a three-edge path is obtained by choosing its middle edge and one
further incident edge at each endpoint, except when the two new endpoints
coincide and form a triangle.  Counting adjacent edge-pairs inside an
unordered triple of edges gives
\[
 N_{P_3\sqcup K_2}(G)
 =(m-2)P-3T-3S-2N_{P_4}(G),
\]
and therefore
\[
 N_{3K_2}(G)=\binom m3-T-S-N_{P_4}(G)
              -N_{P_3\sqcup K_2}(G).
\]
The triangle-containing four-edge supports which occur in $M_6$ are also
fixed:
\[
 N_{\operatorname{Paw}}(G)=3(d-2)T,
 \qquad
 N_{K_3\sqcup K_2}(G)=(m-3d+3)T.
\]
After these quantities are substituted in Theorem~\ref{thm:first-moments},
the only undetermined terms are
\[
 72N_{C_4}(G)+132N_{K_4-e}(G)+96N_{K_4}(G).
\]
Thus the displayed combination after division by $12$ is determined by the
Cayley spectrum.  If $G$ is triangle-free then both $K_4-e$ and $K_4$ are
absent, so $N_{C_4}(G)$ is determined.
\end{proof}

\subsection{The eighth moment}

By Theorem~\ref{thm:rank-bound}, the supports contributing to $M_8$ have
rank at most $4$, so their connected components have at most five
vertices.  We use the following names for connected graphs on five
vertices: $K_{1,4}$, the chair (fork) $\mathrm{Ch}$, the path $P_5$; the
cycle $C_5$, the banner $\mathrm{Ban}$ ($C_4$ with a pendant edge), the
bull, the cricket $\mathrm{Cr}$ ($K_3$ with two pendant edges at one
vertex), the tadpole $T_{3,2}$ ($K_3$ with a pendant path of length
two); $K_{2,3}$, the house $\mathrm{Ho}$, the butterfly $\mathrm{Bf}$
(two triangles sharing a vertex), $D_3$ and $D_2$ ($K_4-e$ with a
pendant edge attached at a vertex of degree $3$, resp.\ $2$); the book
$K_{1,1,3}$, the gem $\mathrm{Gem}$ ($P_4$ joined to a vertex), $K_4$
with a pendant edge $K_4^{+}$, the wheel $W_4$ minus a spoke $W_4^-$;
$K_5-P_3$ and $W_4=K_5-2K_2$.

\begin{theorem}[Eighth moment]
\label{thm:eighth}
For every graph $G$,
\[
 \frac{M_8(G)}{n!}=\sum_F c_8(F)\,N_F(G),
\]
where the nonzero coefficients are as follows.  For connected $F$ with at
most four vertices:
\[
\begin{array}{c|ccccccccc}
F & K_2 & P_3 & K_3 & K_{1,3} & P_4 & \mathrm{Paw} & C_4 & K_4-e & K_4\\\hline
c_8(F) & 1 & 84 & 1932 & 420 & 552 & 3096 & 2832 & 10032 & 16896
\end{array}
\]
For connected $F$ on five vertices (these all satisfy
$k=8=2(|V(F)|-1)$, i.e.\ they are of genus zero):
\[
\begin{array}{c|ccc|ccccc}
F & K_{1,4} & \mathrm{Ch} & P_5 & C_5 & \mathrm{Ban} & \mathrm{Bull} & \mathrm{Cr} & T_{3,2}\\\hline
c_8(F) & 336 & 608 & 816 & 400 & 1008 & 1448 & 1136 & 1952
\end{array}
\]
\[
\begin{array}{c|ccccc|cccc|cc}
F & K_{2,3} & \mathrm{Ho} & \mathrm{Bf} & D_3 & D_2 & K_{1,1,3} & \mathrm{Gem} & K_4^{+} & W_4^- & K_5-P_3 & W_4\\\hline
c_8(F) & 1200 & 1536 & 3712 & 1584 & 2112 & 1920 & 1856 & 1248 & 1344 & 960 & 752
\end{array}
\]
For $K_5-e$ and $K_5$ the coefficient vanishes, since they have more than
eight edges.  For disconnected $F$, Proposition~\ref{prop:exp-formula}
gives the sixteen nonzero values
\[
\begin{array}{c|cccccccc}
F & 2K_2 & 3K_2 & 4K_2 & P_3\sqcup 2K_2 & K_3\sqcup 2K_2 & P_3\sqcup K_2 & K_3\sqcup K_2 & K_{1,3}\sqcup K_2\\\hline
c_8(F) & 126 & 1260 & 2520 & 1680 & 5040 & 840 & 5880 & 840
\end{array}
\]
\[
\begin{array}{c|cccccccc}
F & P_4\sqcup K_2 & \mathrm{Paw}\sqcup K_2 & C_4\sqcup K_2 & (K_4-e)\sqcup K_2 & K_4\sqcup K_2 & 2P_3 & P_3\sqcup K_3 & 2K_3\\\hline
c_8(F) & 1176 & 2688 & 2016 & 3696 & 2688 & 1120 & 3360 & 10080
\end{array}
\]
\end{theorem}

\begin{proof}
By Theorem~\ref{thm:rank-bound} and Proposition~\ref{prop:exp-formula},
it suffices to know $c_j(F)$ for $j\le 8$ and connected $F$ with at most
five vertices.  These were computed by the procedure described in the
remark following Theorem~\ref{thm:first-moments}.  As consistency
checks: the five-vertex coefficients, weighted by $5!/|\operatorname{Aut}F|$,
add up to the Hurwitz number $8!\cdot 5^2=1\,008\,000$ of
Corollary~\ref{cor:hurwitz}; summing $c_8$ over all edge subsets of random
graphs with $5\le n\le 7$ reproduces the exact number of identity words
of length $8$; and the values $c_4,c_6$ from the same computation reproduce
Theorem~\ref{thm:first-moments}.
\end{proof}

\begin{remark}
The same computation in length $10$ gives, for connected
graphs on six vertices,
\[
 c_{10}(K_{1,5})=5040,\qquad c_{10}(P_6)=25520,\qquad c_{10}(C_6)=2100,
\]
and the values of $c_{2e}(T)$ for the trees $T$ with $e\le 5$ edges are
\[
\begin{array}{c|ccccccccc}
T & K_2&P_3&K_{1,3}&P_4&K_{1,4}&\mathrm{Ch}&P_5 \\\hline
c_{2e}(T)&1&4&30&42&336&608&816
\end{array}
\]
\[
\begin{array}{c|cccccc}
T & K_{1,5} & S_{4} & DS & S(3,1,1) & S(2,2,1) & P_6\\\hline
c_{10}(T)& 5040 & 11220 & 14480 & 18880 & 19680 & 25520
\end{array}
\]
Here $S_4$ is $K_{1,4}$ with one edge subdivided, $DS$ is the double
star with degree sequence $(3,3,1,1,1,1)$, and $S(a,b,c)$ is the spider
with legs of lengths $a,b,c$.
\end{remark}

\begin{conjecture}[Genus-zero closed formulae]
\label{conj:closed}
\begin{enumerate}[(i)]
\item For stars, $c_{2e}(K_{1,e})=\dfrac{(2e)!}{(e+1)!}=e!\,\mathrm{Cat}_e$.
\item For cycles, $c_{2(v-1)}(C_v)=v(v-1)^2\,\mathrm{Cat}_{v-2}$.
\item For a tree $T$, a word of length $2|E(T)|$ in the edges of $T$
using every edge has product $1$ if and only if every edge occurs exactly
twice and, for every vertex $x$ of $T$, the subword formed by the edges
incident to $x$ is a non-crossing double occurrence word.
\end{enumerate}
\end{conjecture}

We verified (i) for $e\le5$, (ii) for $v\le6$, and (iii) for all trees
with at most $5$ edges.  Part (iii) implies (i), since for a star the
condition says that the double occurrence word is non-crossing; there
are $\mathrm{Cat}_e$ non-crossing pairings and $e!$ labellings.  Minimal
transitive factorizations into star transpositions have been studied by
Goulden--Jackson and Irving--Rattan \cite{GouldenJackson09,IrvingRattan},
and (i) should follow from their methods.

\begin{remark}
The correction term $12N_{K_3}(G)$ in the fourth moment is important.
It reflects the relation $(ij)(jk)(ij)=(ik)$ inside a triangle.  Thus
triangle effects begin at order four, not order six.  The sixth moment is
still the first place where the next family of local supports, such as
claws, four-vertex paths and four-cycles, enters the regular trace in a
substantial way.
\end{remark}

\section{Three exactly solvable families}
\label{sec:solvable-families}

We next record three classes of graphs for which the transposition
Cayley spectrum can be described explicitly.  These examples answer part
of the general program in closed form and show that the construction is
not merely a source of low-order moment identities.

\subsection{The complete graph}

Let
\[
 C_n=\sum_{1\leq i<j\leq n}(ij)
\]
be the full class sum of all transpositions.  This is a central element
of $\mathbb C[S_n]$.  Its spectral decomposition is the standard
character-theoretic decomposition of the random-transposition walk on
$S_n$; see, for example, Diaconis--Shahshahani and Diaconis
\cite{DiaconisShahshahani,Diaconis}.  Hence, by Schur's lemma, it acts
as a scalar on each irreducible representation $V_\lambda$.

\begin{proposition}[Complete transposition graph]
\label{prop:complete-graph-spectrum}
For $G=K_n$, the eigenvalue of $A_G=C_n$ on the irreducible
representation $V_\lambda$ is
\[
 \kappa_\lambda
 =
 \sum_{(a,b)\in\lambda}(b-a),
\]
the sum of the contents of all boxes of the Young diagram $\lambda$.
Equivalently,
\[
 \kappa_\lambda
 =
 \frac12\sum_i \lambda_i(\lambda_i-2i+1).
\]
In the regular representation this eigenvalue occurs with multiplicity
$(f^\lambda)^2$.
\end{proposition}

\begin{proof}
The element $C_n$ is central.  Its scalar on $V_\lambda$ is the usual
central character of the transposition class.  In the Young--Jucys--
Murphy realization one has
\[
 C_n=J_2+J_3+\cdots+J_n,
 \qquad
 J_r=\sum_{i<r}(ir).
\]
The element $J_r$ acts diagonally in the Young orthogonal basis, with
eigenvalue equal to the content of the box containing $r$.  Summing over
$r$ gives the sum of all contents of $\lambda$.  The multiplicity
$(f^\lambda)^2$ follows from the decomposition of the regular
representation.
\end{proof}

\subsection{Matchings}

Suppose that $G$ is a matching with $m$ edges.  Then the corresponding
transpositions commute.  Thus $A_G$ is a sum of $m$ commuting involutions,
and its spectrum is completely elementary.

\begin{proposition}[Matching case]
\label{prop:matching-spectrum}
If $G=mK_2$ is a matching with $m$ edges on $n$ vertices, then
\[
 \operatorname{Spec}(\Gamma(G))
 =
 \left\{\,m-2r:\;0\leq r\leq m\,\right\},
\]
where the eigenvalue $m-2r$ occurs with multiplicity
\[
 n!\,2^{-m}\binom mr.
\]
\end{proposition}

\begin{proof}
The subgroup generated by the $m$ disjoint transpositions is
$(\mathbb Z/2\mathbb Z)^m$.  In a simultaneous eigenbasis for these
commuting involutions, each transposition contributes an eigenvalue
$\pm1$.  If exactly $r$ of the signs are negative, the sum is $m-2r$,
and there are $\binom mr$ such sign choices.  Each character of this
elementary abelian subgroup occurs in the regular representation of
$S_n$ with multiplicity $n!/2^m$.
\end{proof}

\subsection{Stars and a Jucys--Murphy element}

The first noncentral but still completely diagonalizable family is given
by stars.  This family is the classical star graph on $S_n$; its spectrum
and eigenvalue multiplicities have been studied, for instance, by
Krakovski--Mohar and by Khomyakova--Konstantinova
\cite{KrakovskiMohar,KhomyakovaKonstantinova}.  Let $G=K_{1,n-1}$ be the
star with centre $n$.  Then
\[
 A_G=\sum_{i=1}^{n-1}(in)=J_n,
\]
the last Young--Jucys--Murphy element.

\begin{theorem}[Star spectrum]
\label{thm:star-spectrum}
Let $G=K_{1,n-1}$.  Then the eigenvalues of $A_G$ are the possible
contents
\[
 c(b)=\operatorname{column}(b)-\operatorname{row}(b)
\]
of removable boxes $b$ of Young diagrams $\lambda\vdash n$.
More precisely, in the copy of $V_\lambda$ with Young orthogonal basis,
$A_G$ is diagonal and the basis vector indexed by a standard tableau $T$
has eigenvalue equal to the content of the box occupied by $n$ in $T$.
Thus the multiplicity of an integer $c$ in the regular representation is
\[
 \sum_{\lambda\vdash n}
 f^\lambda\,
 \#\{\,T\in \operatorname{SYT}(\lambda):
      c(\operatorname{box}_T(n))=c\,\}.
\]
\end{theorem}

\begin{proof}
After choosing the centre of the star to be $n$, the operator $A_G$ is
exactly the Young--Jucys--Murphy element $J_n=\sum_{i<n}(in)$.  The
Young--Jucys--Murphy elements are diagonal in the Young orthogonal basis,
and $J_n$ acts on a standard tableau $T$ by the content of the box
containing $n$.  The stated multiplicity again follows by multiplying
the multiplicity inside $V_\lambda$ by the multiplicity $f^\lambda$ of
$V_\lambda$ in the regular representation.
\end{proof}

\begin{remark}
The star theorem gives a genuine solution of one nontrivial instance of
the inverse problem.  Among stars, the Cayley spectrum determines the
number of leaves, hence the graph.  It also shows why the present
problem is naturally related to the Gelfand--Tsetlin and
Young--Jucys--Murphy approach to the representation theory of symmetric
groups.
\end{remark}

\begin{corollary}[A solved family of reconstruction examples]
\label{cor:solved-families}
The reconstruction problem has an affirmative answer inside each of the
following families:
\[
 \{K_n\}_{n\geq1},\qquad
 \{mK_2\}_{m\geq0},\qquad
 \{K_{1,n-1}\}_{n\geq1}.
\]
In each case the transposition Cayley spectrum determines the parameter
of the family and hence the graph.
\end{corollary}

The three examples above indicate that the general problem interpolates
between three regimes: the central character theory of the complete
graph, the elementary Fourier theory of an abelian subgroup for
matchings, and the noncentral but diagonal Young--Jucys--Murphy theory
for stars.  This gives a conceptual explanation for why the general
operator $A_G$ should be viewed as a graph-dependent deformation of the
Young--Jucys--Murphy elements.

\section{The reconstruction problem}

The main conjectural direction is the following.

\begin{conjecture}[Reconstruction conjecture]
\label{conj:reconstruction}
If
\[
\operatorname{Spec}(\Gamma(G_1))
=
\operatorname{Spec}(\Gamma(G_2)),
\]
then
\[
G_1\cong G_2.
\]
\end{conjecture}

The corresponding statement for ordinary graph spectra is false,
since there exist many pairs of cospectral nonisomorphic graphs.
However, the spectrum of $\Gamma(G)$ is expected to be substantially
more rigid due to the interaction of the transpositions inside
the representation theory of $S_n$.  Conjecture~\ref{conj:reconstruction}
has now been verified by computer for all graphs with at most nine
vertices; see Section~\ref{sec:computations}.

\subsection{Connectedness and component-size information}

The reconstruction conjecture is difficult, but one global part of it can
be settled immediately.  The Cayley spectrum detects whether the original
graph is connected.  In fact it detects a little more: it determines the
index in $S_n$ of the subgroup generated by the edge-transpositions.

\begin{theorem}[Connectedness test]
\label{thm:connectedness-test}
Let $G$ be a graph on $n$ vertices with connected components of sizes
\[
 s_1,\ldots,s_r.
\]
Let $m=|E(G)|$.  Then the multiplicity of the eigenvalue $m$ in the
spectrum of $\Gamma(G)$ is
\[
 \frac{n!}{s_1!s_2!\cdots s_r!}.
\]
Consequently the spectrum of $\Gamma(G)$ determines whether $G$ is
connected.
\end{theorem}

\begin{proof}
The Cayley graph $\Gamma(G)$ is $m$-regular.  For a finite undirected
$m$-regular graph, the multiplicity of the eigenvalue $m$ is equal to the
number of connected components of that graph.  The connected components
of a Cayley graph are the left cosets of the subgroup generated by its
connection set.  In the present case this subgroup is generated by all
transpositions corresponding to edges of $G$.

If $C_1,\ldots,C_r$ are the connected components of $G$, then the edge
transpositions generate
\[
 S_{C_1}\times S_{C_2}\times\cdots\times S_{C_r}
 \subseteq S_n,
\]
because transpositions along the edges of a connected graph generate the
full symmetric group on its vertex set.  Therefore the number of connected
components of $\Gamma(G)$ is the index
\[
 [S_n:S_{C_1}\times\cdots\times S_{C_r}]
 =\frac{n!}{s_1!s_2!\cdots s_r!}.
\]
In particular this multiplicity is equal to $1$ if and only if $r=1$,
which is exactly the connectedness of $G$.
\end{proof}

\begin{corollary}[Necessary conditions for Cayley cospectrality]
\label{cor:necessary-conditions}
If $\Gamma(G_1)$ and $\Gamma(G_2)$ are cospectral, then $G_1$ and $G_2$
have the same number of vertices, the same number of edges, the same
value of $6N_{K_3}-N_{P_3}$, and the same value of
\[
 \prod_C |C|!,
\]
where the product is over the connected components $C$ of the graph.  In
particular, either both $G_1,G_2$ are connected or neither is connected.
If they are connected, they also have the same algebraic connectivity
$a(G_1)=a(G_2)$, by Corollary~\ref{cor:aldous}.
\end{corollary}

\begin{proof}
The number of vertices is determined by the size $n!$ of the Cayley
matrix, hence by the total multiplicity of the spectrum.  The number of
edges and the invariant $6N_{K_3}-N_{P_3}$ are determined by
Corollary~\ref{cor:fourth-moment-invariant}.  The product of factorials of
component sizes is determined from Theorem~\ref{thm:connectedness-test},
since the multiplicity of the top eigenvalue is
$n!/\prod_C |C|!$.
\end{proof}

\section{Relation with the CayleyPy programme}

The CayleyPy papers suggest a useful way to sharpen the computational
side of the present project.  Their main focus is not the inverse problem
studied here, but rather the efficient exploration of very large Cayley
and Schreier graphs, shortest paths, diameters, growth vectors, random
walk behaviour, and experimentally generated conjectures
\cite{ChervovCube,ChervovRL,ChervovGrowth,ChervovHolography}.  Still,
there is a clear intersection at the level of spectra.  The CayleyPy-RL
paper reports numerical spectral-distribution phenomena for the
``LRX'' Cayley graphs of $S_n$, generated by a cyclic shift and a
transposition \cite{ChervovRL}.  The CayleyPy-Growth paper asks, among
other general questions, what can be said about spectra of Cayley graphs
and provides a large experimental framework for computing growth and
related statistics \cite{ChervovGrowth}.

The present paper is narrower but more algebraic.  We do not study an
arbitrary generating set of $S_n$; instead the generating set is the edge
set of an arbitrary graph $G$, viewed as a set of transpositions.  This
restriction makes it possible to express the regular spectral moments of
$\Gamma(G)$ as universal linear combinations of small subgraph counts in
$G$.  Thus the two approaches should be regarded as complementary.  The
CayleyPy framework is a natural tool for testing the reconstruction
question, for producing non-isomorphic candidates with equal low-order
moments, and for comparing growth or diameter data with the spectral
moments studied here.  Conversely, the support expansion in this note
supplies exact invariants which could be used as test cases or symbolic
benchmarks for CayleyPy computations.

\section{Computational results}
\label{sec:computations}

\subsection{The CayleyPy computation for $n\le 9$}

Following item~(1) of the list of finite tests in the previous version
of this paper, I.~Koltsov of the CayleyPy project computed the refined
spectra of all graphs with $n\le 9$ \cite{KoltsovPC}.  The datasets of
non-isomorphic graphs were generated with \texttt{nauty} (from
B.~McKay's tables), and the matrices $A_G^\lambda$ were built in Young's
orthogonal form.  This avoids the $n!\times n!$ regular representation.
His findings are as follows.
\begin{enumerate}[(1)]
\item For all $n\le 9$ (for $n=9$ there are $274\,668$ graphs), the
refined spectra computed to eight decimal places show no collisions.
Thus $G\mapsto\{\operatorname{Spec}A_G^\lambda\}_{\lambda\vdash n}$ is
injective on isomorphism classes for $n\le 9$.
\item There are no collisions either for the combined spectrum, i.e.\
the multiset of eigenvalues with multiplicities, with the labels
$\lambda$ forgotten.  This combined spectrum is exactly
$\operatorname{Spec}\Gamma(G)$.  Hence
\emph{Conjecture~\ref{conj:reconstruction} holds for all graphs with at
most nine vertices.}
\item If multiplicities are also discarded, so that only the support of
the spectrum (the set of distinct eigenvalues) is kept, collisions do
occur.  There are few among connected graphs and considerably more among
all graphs.
\item The YOF pipeline was cross-checked against direct diagonalization
of the $n!\times n!$ regular adjacency matrix for all graphs with
$n\le 6$, with identical sets of distinct eigenvalues.
\end{enumerate}

\subsection{An independent verification}

We independently checked these findings with separate code, using
Young's orthogonal form built from the axial-distance formula and graphs
generated by \texttt{geng} \cite{McKayPiperno}.  The implementation was
validated against brute-force regular spectra for $n\le5$, and against
the identities $s_i^2=1$, the braid relations and
$\sum_\lambda(f^\lambda)^2=n!$ for $n\le7$.  The results are summarized
in Table~\ref{tab:census}.

\begin{table}[h]
\[
\begin{array}{c|r|cccccc|c|c}
n & \#\text{graphs} &
\multicolumn{6}{c|}{\#\text{classes of }(M_2,\dots,M_{2r})} &
\text{support coll.} & \text{adj.\ cospectral}\\
 & & r=1 & r=2 & r=3 & r=4 & r=5 & & \text{(all / conn.)} & \text{graphs}\\\hline
5 & 34 & 11 & 32 & 34 & & & & 4\ /\ 2 & 2\\
6 & 156 & 16 & 111 & 156 & & & & 23\ /\ 8 & 10\\
7 & 1044 & 22 & 341 & 1022 & 1044 & & & 89\ /\ 39 & 110\\
8 & 12346 & 29 & 856 & 10118 & 12341 & 12346 & & 343\ /\ 156 & 1722\\
9 & 274668 & 37 & 7308 & 234228 & 274642 & 274668 & & \text{--} & 51039\\
\end{array}
\]
\caption{Separation of graphs by Cayley moments.  The $r$-th column gives
the number of distinct values of $(M_2,\dots,M_{2r})$; the graphs are
separated once this equals the number of graphs.  ``Support coll.'' is
the number of graphs (all, resp.\ connected) whose set of distinct
Cayley eigenvalues coincides with that of a non-isomorphic graph.  The
last column gives the number of graphs having an adjacency-cospectral
mate.}
\label{tab:census}
\end{table}

\begin{theorem}
\label{thm:computer}
\textup{(Computer-assisted.)}  For $n\le 9$, non-isomorphic graphs on
$n$ vertices have distinct moment vectors $(M_2,M_4,\dots,M_{10})$.  In
particular their transposition Cayley graphs are not cospectral, which
confirms Conjecture~\ref{conj:reconstruction} for $n\le 9$.  The minimal
number of even moments needed is $3$ for $n=5,6$, $4$ for $n=7$, and $5$
for $n=8,9$.
\end{theorem}

\begin{proof}[Method]
For each graph we computed the refined spectrum numerically and the
normalized moments $M_{2j}(G)/n!$, which are integers.  For $j\le4$ the
floating-point values were within $10^{-7}$ of integers, so the
partition into classes by $(M_2,\dots,M_8)$ is exact.  For $n=8$ this
leaves five pairs of graphs, and for $n=7$ the classes of
$(M_2,M_4,M_6)$ leave $21$ classes.  In each such case we computed the
next moment \emph{exactly}, as the number $W_{2j}(E(G))$ of identity
words, by integer dynamic programming on $S_n$, and found that it
separates the graphs.

For $n=9$ the normalized moments of dense graphs are too large (of order
$36^{10}$) for exact rounding of floating-point values.  We therefore
regarded two graphs as possibly equal in $M_{2j}$ when the computed
values agreed to relative precision $10^{-10}$.  The numerical error is
of order $10^{-12}$, so graphs that differ at this precision are
certainly different.  Exactly $26$ pairs of graphs agree in
$M_2,\dots,M_8$ at this precision, and no two graphs agree in
$M_2,\dots,M_{10}$.  For each of the $26$ pairs we computed $W_8$ and
$W_{10}$ exactly on $S_9$: in every pair $W_8$ coincides and $W_{10}$
differs.  The graphs were generated by \texttt{geng}, so the list is
complete.
\end{proof}

Theorem~\ref{thm:computer} confirms, by an independent method, the
CayleyPy findings (1)--(2) for $n\le 9$, since separation by the regular
spectrum implies separation by the refined spectrum.  We also recomputed the support
collisions for $n\le 8$; their existence, and the fact that they are
rarer among connected graphs, agree with finding~(3).  The smallest
example is $n=5$: the diamond with a pendant edge at a vertex of degree
$3$, and the butterfly, have the same set of distinct Cayley eigenvalues
but different multiplicities.  We also recounted the
adjacency-cospectral graphs (last column of Table~\ref{tab:census}),
and our counts agree with those of Haemers and Spence \cite{HaemersSpence}.

\subsection{What the first moments do not see}

The following pairs show that the low moments alone are far from
determining the basic invariants discussed in Section~\ref{sec:first-moments}.
Edges are written $ij$ for $\{i,j\}$.

\begin{proposition}[Limits of the low moments]
\label{prop:counterexamples}
The moments $M_2,M_4,M_6$ do not determine the number of triangles, the
degree sequence, or the connectedness of $G$.  The moments
$M_2,M_4,M_6,M_8$ do not determine the number of triangles or the
degree sequence.  Explicitly, (a) and (b) are on the vertex set $[7]$,
and (c) is on $[9]$:
\begin{enumerate}[(a)]
\item $G_1=\{15,16,17,26,27,36,37,46,47,57,67\}$ and
$G_2=\{14,16,17,25,26,27,35,37,46,47,57\}$ are connected, with
$M_{2j}/7!=11,337,15837$ for $j=1,2,3$ for both.  However
$N_{K_3}(G_1)=5\neq4=N_{K_3}(G_2)$, and the degree sequences are
$(6,5,3,2,2,2,2)$ and $(5,3,3,3,3,3,2)$.  Their eighth moments are
$948029$ and $958353$.
\item $G_3=\{15,17,26,27,36,37,46,67\}$ is connected, and
$G_4=\{15,16,25,27,36,37,57,67\}$ is not (vertex $4$ is isolated).  Both
have $M_{2j}/7!=8,170,5186$ for $j=1,2,3$ and two triangles.  Their
eighth moments are $191786$ and $192546$.
\item The graphs
\[
\begin{aligned}
G_5&=\{16,17,19,26,28,29,37,38,39,48,49,59,69,79\},\\
G_6&=\{15,17,26,27,28,36,38,39,48,49,59,67,69,79\}
\end{aligned}
\]
are connected, with
$M_{2j}/9!=14,528,30010,2176660$ for $j=1,2,3,4$ for both.  However
$N_{K_3}(G_5)=4\neq3=N_{K_3}(G_6)$, and the degree sequences are
$(7,3,3,3,3,3,3,2,1)$ and $(5,4,4,3,3,3,2,2,2)$.  Their tenth moments
are $186799374$ and $187065334$.
\end{enumerate}
\end{proposition}

\begin{proof}
Direct computation of the numbers of identity words by integer dynamic
programming on $S_n$.  The values agree with those obtained from the
numerically computed refined spectra.
\end{proof}

In (a) and (b) the eighth moment, and in (c) the tenth moment,
separates the two graphs, and hence so does the full spectrum.  In (b)
this also follows from Theorem~\ref{thm:connectedness-test}.  So these
examples
show that the moment information in Section~\ref{sec:first-moments} is
strictly weaker than the spectrum, and that the triangle problem below
cannot be settled using $M_2,\dots,M_8$ alone.

\begin{remark}[Pairs agreeing up to $M_8$]
For $n=8$ there are exactly five pairs of non-isomorphic graphs with
equal $M_2,M_4,M_6,M_8$, and for $n=9$ there are exactly $26$ such pairs,
all connected.  For $n=8$, in all five pairs the two graphs have the same
degree sequence and the same number of triangles.  For $n=9$ this holds
for $23$ of the $26$ pairs; the exceptions include pair~(c) above.  For $n=8$, four of the five
pairs are Laplacian cospectral, so they also have the same
$(n-1,1)$-component of the refined spectrum.  None of them is adjacency
cospectral.  An example is the pair of triangle-free graphs
\[
\{15,17,18,25,27,28,36,37,38,46,47,48\},\quad
\{15,16,17,18,25,28,36,37,38,46,47,48\},
\]
with $M_{10}/8!=60025212$ and $60306492$ respectively.  Thus $M_2,\dots,M_8$ determine the number of triangles and the degree
sequence for $n\le 8$, but not for $n=9$.
\end{remark}

\section{Consequences and remaining problems}

The results above give a finite-type spectral theory for the first
Cayley moments.  They also indicate several precise ways in which the
paper can be pushed further.

\subsection{Triangle separation}

The fourth moment determines $6N_{K_3}(G)-N_{P_3}(G)$.  Thus triangle
relations are already visible in the Cayley spectrum at order four.
However, this invariant mixes triangles with adjacent edge-pairs, and by
Proposition~\ref{prop:counterexamples} the moments $M_2,M_4,M_6$ do
\emph{not} determine $N_{K_3}(G)$.  A sharp next goal is therefore the
following.

\begin{question}[Triangle problem]
Is the number of triangles $N_{K_3}(G)$ determined by the full Cayley
spectrum of $\Gamma(G)$?  More modestly, is it determined by finitely
many even moments $M_2,M_4,\ldots,M_{2r}$?  By
Proposition~\ref{prop:counterexamples}(c), $r=4$ is not enough.
\end{question}

For $n\le9$ the first part has a positive answer by
Section~\ref{sec:computations}.  A positive answer in general would be
the first genuinely nonlinear reconstruction result separating the
triangle contribution from the adjacent-edge-pair contribution in the
fourth moment.

\subsection{Degree sequence problem}

On triangle-free graphs the fourth moment determines $\sum_v d_v^2$.
More generally, supports which are stars should control higher power sums
of the degrees, but the general case requires separating star
contributions from cycle and clique contributions.  This leads to the
following concrete conjecture.

\begin{conjecture}[Degree-power conjecture]
\label{conj:degree}
For each $q\geq 1$, the Cayley spectrum determines the power sum
\[
 p_q(G)=\sum_{v\in V(G)}d_v^q.
\]
Consequently, the Cayley spectrum determines the degree sequence of
$G$.
\end{conjecture}

The case $q=1$ follows from $M_2$.  For triangle-free graphs the case
$q=2$ follows from $M_4$, and Corollary~\ref{cor:triangle-free-sixth}
shows that the sixth moment determines a first combination involving the
claw $K_{1,3}$.  The general case remains open because clique and cycle
relations mix with the star counts.  By
Proposition~\ref{prop:counterexamples}(a), the moments $M_2,M_4,M_6$ do
not determine $p_2(G)$ for general graphs: for that pair, $p_2=86$ and
$p_2=74$.  By part (c) of the same proposition, $M_2,\dots,M_8$ do not
determine the degree sequence either.  For graphs with $n\le 9$ the conjecture holds by
Section~\ref{sec:computations}.

\subsection{Irreducible spectral data}

Proposition~\ref{prop:standard-laplacian} shows that the refined
irreducible spectrum contains the Laplacian spectrum in the standard
component.  Thus a useful intermediate invariant between the ordinary
Cayley spectrum and the graph $G$ is the collection
\[
 \left\{
 \operatorname{Spec}\left(\sum_{e\in E(G)}\rho_\lambda(\tau_e)\right)
 : \lambda\vdash n
 \right\}.
\]
This refined spectrum is representation-theoretically natural and may
be more accessible computationally than the undifferentiated regular
spectrum.

\begin{question}
Does the refined irreducible spectrum reconstruct $G$?  If not, does it
at least reconstruct all subgraph counts of $G$?
\end{question}

By the CayleyPy computation the answer to the first question is positive
for $n\le 9$.  Theorem~\ref{thm:token} identifies the two-row part of
the refined spectrum with the token-graph Laplacian spectra
$\operatorname{LSpec}F_k(G)$, $k\le n/2$.  The first new member of this
hierarchy, $(n-2,2)$, already has substantial reconstructive content.
Reference~\cite{ShapiroN22} proves that for a tree $T$ the pair consisting
of the ordinary Laplacian spectrum and the $(n-2,2)$-spectrum determines
separately all four three-edge support-forest counts
\[
 K_{1,3},\qquad P_4,\qquad P_3\sqcup K_2,\qquad 3K_2,
\]
and reports that this pair separates all Laplacian-cospectral trees in the
tested range through $15$ vertices.  Its weighted trace version is much
stronger still, as noted after Theorem~\ref{thm:weighted-fourth}.
A natural intermediate question is therefore whether the family of
Laplacian spectra of all token graphs $F_k(G)$ already determines $G$.

\subsection{Further computational tests}

Item (1) of the list of tests proposed in the previous version (full and
refined spectra of all graphs with at most eight vertices) has been
carried out, up to $n=9$; see Section~\ref{sec:computations}.  Since all
adjacency-cospectral pairs with $n\le 9$ are among the graphs checked,
item (2) is settled in this range as well.  The following tests remain
natural:
\begin{enumerate}[(1)]
\item test Godsil--McKay switched pairs and strongly regular cospectral
families with $n\ge 10$, where exhaustive enumeration is no longer
feasible;
\item determine, for each $n$, the smallest $r(n)$ such that
$M_2,\dots,M_{2r(n)}$ separate all graphs on $n$ vertices
($r(n)=3,3,4,5,5$ for $n=5,\dots,9$), and decide whether $r(n)$ is bounded;
\item use the eighth-moment formula (Theorem~\ref{thm:eighth}) and the
corresponding $M_{10}$ formula (which needs only connected supports on at
most six vertices) to search for triangular combinations
isolating $N_{K_3}(G)$ and then $N_{K_{1,3}}(G)$;
\item prove Conjecture~\ref{conj:closed} and determine the genus-zero
coefficients $c_{2(v-1)}(F)$ for all connected $F$.
\end{enumerate}

\section{Final remarks}

The main point of this paper is that the transposition Cayley spectrum
admits a tractable moment theory.  The regular trace converts spectral
moments into counts of identity words in edge-transpositions, while the
irreducible decomposition of $\mathbb C[S_n]$ gives a parallel
representation-theoretic formula.  These two descriptions suggest that
$\operatorname{Spec}(\Gamma(G))$ is a much finer invariant than the
ordinary adjacency spectrum of $G$.

The computations through order six are only the first step, but they
already show that nontrivial local data of $G$ is visible in the Cayley
spectrum.  The fourth moment detects the signed combination
$6N_{K_3}-N_{P_3}$, while the triangle-free sixth-moment specialization
recovers the next combination $5N_{K_{1,3}}+2N_{P_4}+12N_{C_4}$.  The
most natural next goal is to determine whether higher moments, possibly
together with irreducible components of the spectrum, can separate these
combined local invariants and then progressively recover all small
subgraph counts.  The rank filtration and the genus-zero Hurwitz
structure (Section~3) show that this is a problem about refined Hurwitz
numbers.  The computations of Section~\ref{sec:computations} show that
the reconstruction conjecture holds in every case that can currently be
checked exhaustively.

\vspace{1em}

\noindent
\textbf{Acknowledgements.}
The author is grateful to Ivan Koltsov and the CayleyPy team for
carrying out the computations of refined spectra for all graphs with at
most nine vertices and for sharing their results.

\end{document}